\documentclass[12pt]{amsart}
\usepackage{mathrsfs,amsmath,amssymb,amsfonts,amsthm,latexsym,bm}
\theoremstyle{plain}
\newtheorem{theorem}{Theorem}[section]
\newtheorem{proposition}[theorem]{Proposition}
\newtheorem{corollary}[theorem]{Corollary}
\newtheorem{lemma}[theorem]{Lemma}

\newtheorem*{theorem2}{Question}

\theoremstyle{definition}
\newtheorem{remark}[theorem]{Remark}
\newtheorem{example}[theorem]{Example}

\title[Commuting and Semi-commuting Operators]{On Commuting and Semi-commuting Positive Operators}
\author[N.~Gao]{Niushan Gao}
\address{Department of Mathematical and Statistical Sciences, University of Alberta,
Edmonton, AB, Canada T6G\,2G1}
\email{niushan@ualberta.ca}
\date{\today}

\begin{document}

\begin{abstract}Let $K$ be a positive compact operator on a Banach lattice. We prove that if either $[K\rangle$ or $\langle K]$ is ideal irreducible then $[K\rangle=\langle K]=L_+(X)\cap \{K\}'$. We also establish the Perron-Frobenius Theorem for such operators $K$.
Finally we apply the results to answer questions in \cite{abr} and \cite{jbr}.
\end{abstract}
\maketitle

\section{Introduction}
Throughout this paper, $X$ always denotes a real Banach lattice with $\dim X>1$. Recall that a collection $\mathcal{C}$ of positive operators on $X$ is said to be \textbf{ideal irreducible} if 
there exist no non-trivial (that is, different from $\{0\} $ and
$X$) closed ideals invariant under each member of $\mathcal{C}$. In particular, a positive operator $T$ is called ideal irreducible if $\{T\}$ is ideal irreducible.
The classical Perron-Frobenius theory studies the peripheral spectrum of ideal irreducible operators on finite-dimensional spaces; see, e.g.~\cite[Chapter 8]{abr}. It has been extended to ideal irreducible
operators on infinite-dimensional Banach lattices by various authors; see, e.g.~\cite{scha2,sawa,ns,grob1,grob2,kit}, etc. In particular, one has the following.

\begin{theorem}[\cite{pag,ns}; cf.~Theorems~5.2 and 5.4, Chapter V of \cite{scha3}]\label{dsns}Let $K>0$ be compact and ideal irreducible.
\begin{enumerate}
\item\label{dsnsi1} The spectral radius $r(K)>0$, $\ker (r(K)-K)=\mathrm{Span}\{x_0\}$ for some quasi-interior point $x_0>0$ and $\ker (r(K)-K^*)=\mathrm{Span}\{x_0^*\}$ for some strictly positive functional $x_0^*$;
\item\label{dsnsi2} The peripheral spectrum $\sigma_{per}(K)=r(K)G$ where $G$ is the set of all $k$-th roots of unity for some $k\geq 1$, and each point in $\sigma_{per}(K)$ is a simple pole of the resolvent 
$R(\cdot,K)$ with one-dimensional eigenspace.
    \end{enumerate}
\end{theorem}

The fact that $r(K)>0$ is known as de Pagter's theorem; cf.~\cite{pag}.
It was extended by Abramovich, Aliprantis and Burkinshaw~\cite{sab,sab2,sab1} to commuting and semi-commuting positive operators that possess ideal irreducibility and compactness in one or 
another sense (cf.~Chapters 9 and 10 of \cite{abr}).
We write $T\leftrightarrow S$ if $T$ commutes with $S$.
The following is Corollary~9.28 of \cite{abr}.
\begin{theorem}[\cite{abr}]\label{aab1}Let $T,S,K>0$ be such that $T\leftrightarrow S\leftrightarrow K$, $T$ is ideal irreducible and $K$ is compact. Then $r(S)>0$.
\end{theorem}
The following can be found in Theorems~10.25 and 10.26 in \cite{abr}.

\begin{theorem}[\cite{abr}]\label{aab2}Let $T,K>0$ be such that $K$
is compact. Then $T$ has a non-trivial invariant closed ideal if one of the following is satisfied:
\begin{enumerate}
\item\label{aab2i1} $TK\geq KT$ and $\liminf_n||K^nx||^{\frac{1}{n}}=0$ for some $x>0$;
\item\label{aab2i2} $TK\leq KT$ and $K$ is quasi-nilpotent.
\end{enumerate}
\end{theorem}

Extensions of results of this style to collections of positive operators have also been considered by various authors; see, e.g.~\cite{drn2001,at,hali}, etc. 
For a positive operator $T$, we define as in \cite{abr} the \textbf{super right-commutant} of $T$ by $[T\rangle:=\{S\geq 0:ST\geq TS\}$ and the \textbf{super left-commutant} by
$\langle T]:=\{S\geq 0:ST\leq TS\}$. The following can be deduced from Theorem~4.3 of \cite{drn2001}.
\begin{theorem}[\cite{drn2001}]\label{drnl}Let $K>0$ be compact. If $\lim_n||K^nx||^{\frac{1}{n}}=0$ for some $x>0$, then $[K\rangle$ has a common non-trivial invariant closed ideal.
\end{theorem}
\medskip

This paper aims to study the properties of compact positive operators $K$ such that either $[K\rangle$ or $\langle K]$ is ideal irreducible.
In Section~\ref{sec3}, we establish the analogous version of Theorem~\ref{dsns} for such operators $K$. We also prove that, in this case,
$[K\rangle=\langle K]=L_+(X)\cap\{K\}'$ and $\lim_n||K^nx||^{\frac{1}{n}}=r(K)>0$ for all $x>0$.
We also prove existence of positive eigenvectors of positive operators $S$ in the following three chains: $T\leftrightarrow K\leftrightarrow S$, $S\leftrightarrow T\leftrightarrow K$ and $T\leftrightarrow S\leftrightarrow K$,
where $T>0$ is ideal irreducible and $K>0$ is compact.\par

In Section~\ref{sec4}, we provide some applications of the results in Section~\ref{sec3}. In particular, we prove that, for a compact operator $K>0$, if $TK\leq KT$ and $\liminf_n||K^nx||^{\frac{1}{n}}=0$ for some $x>0$, then $T$ has a
non-trivial invariant closed ideal. This improves Theorem~\ref{aab2}\eqref{aab2i2} and answers a question asked in \cite{abr}. As another application, we prove that, for two positive operators 
$T$ and $K$ with $K$ being compact, if they semi-commute then their commutator is quasi-nilpotent. This answers a question proposed in \cite{jbr}.

\section{Preliminaries}\label{sec2}

The notations and terminology used in this paper are standard. We refer the reader to \cite{abr} for unexplained terms.
For $T\in L(X)$, we write $\sigma(T)$ for the spectrum of $T$, $r(T)$ for the spectral radius of $T$ and $\sigma_{per}(T):=\{\lambda\in \sigma(T):|\lambda|=r(T)\}$ for the peripheral spectrum of $T$. If $T$ is a non-zero positive operator, we write $T>0$. A positive operator is called \textbf{strictly positive} if it does not vanish on non-zero positive vectors. We say that two operators \textbf{semi-commute} if their commutator is positive or negative. 
For $A\subset X$, we write $I_A$ for the ideal generated by $A$ in $X$.
A vector $x>0$ is called a \textbf{quasi-interior point} if its generated ideal 
$I_x$ is norm dense in $X$, or equivalently, $\widehat{x} $ acts as a strictly positive functional on $X^*$.
\par
\medskip
We will need some technical lemmas. An operator $T\in L(X)$ is called \textbf{peripherally Riesz} if $r(T)>0$ and $\sigma_{per}(T)$ is a spectral set with finite-dimensional spectral subspace.
The following fact can be deduced by applying Lemma~1 of \cite{radj} to the restriction of $T$ to its spectral subspace for $\sigma_{per}(T)$.
\begin{lemma}[\cite{gt2}]\label{pRiesz}Let $T\in L(X)$ be peripherally Riesz and $r(T)=1$. Then either there exists $n_j\uparrow \infty$ such that $T^{n_j}$ converges to the spectral projection of $T$ for $\sigma_{per}(T)$, 
or there exist $n_j\uparrow \infty$ and positive reals $c_j\downarrow 0$ such that $c_jT^{n_j}$ converges to a non-zero finite-rank nilpotent operator.
 \end{lemma}

The following easy lemma is taken from \cite{gao}; we provide the proof for the convenience of the reader.
\begin{lemma}[\cite{gao}]\label{scc}Let $U,V\in L(X)$ be semi-commuting. Suppose $Ux_0=\lambda x_0$ and $U^*x_0^*=\lambda x_0^*$ for some quasi-interior point
$x_0>0$ and some strictly positive functional $x_0^*$. Then $UV=VU$ 
\end{lemma}
\begin{proof}
Note that $x_0^*((UV-VU)x_0)=(U^*x_0^*)(Vx_0)-x_0^*(VUx_0)=\lambda x_0^*(Vx_0)-x_0^*(V\lambda x_0)=0$. Since $x_0^*$ is strictly positive, we have $(UV-VU)x_0=0$. 
Thus $x_0 $ being quasi-interior yields $UV-VU=0$.\end{proof}

\begin{lemma}\label{locaq}Suppose $T>0$.
\begin{enumerate}
 \item\label{locaqi1}If $Tx_0=\lambda x_0$ for some quasi-interior $x_0>0$ then $\liminf_n||T^{n*}x^*||^{\frac{1}{n}}\geq \lambda$ for any $x^*>0$. If, in addition,
$r(T)$ is an eigenvalue of $T^*$, then $\lambda=r(T)$;
 \item\label{locaqi2}If $T^*x_0^*=\lambda x_0^*$ for some strictly positive $x_0^*>0$ then $\liminf_n||T^{n}x||^{\frac{1}{n}}\geq \lambda$ for any $x>0$.
If, in addition, $r(T)$ is an eigenvalue of $T$, then $\lambda=r(T)$.
\end{enumerate}
\end{lemma}
\begin{proof}We only prove \eqref{locaqi1}; the proof of \eqref{locaqi2} is similar. Suppose that $x^*>0$. Then $x_0$ being quasi-interior implies $x^*(x_0)>0$. Note that $\lambda^nx^*(x_0)=x^*(T^nx_0)=T^{n*}x^*(x_0)\leq ||T^{n*}x^*||\,||x_0||$.
Thus $||T^{n*}x^*||^{\frac{1}{n}}\geq \lambda \sqrt[n]{x^*(x_0)/||x_0||}$. Letting $n\rightarrow \infty$, we have $\liminf_n||T^{n*}x^*||^{\frac{1}{n}}\geq \lambda$.\par
Now suppose $T^*x^*=r(T)x^*$ for some $x^*\neq 0$. Then $\lambda |x^*|\leq r(T)|x^*|\leq T^*|x^*|$. Note also that $(T^*|x^*|-\lambda |x^*|)(x_0)=|x^*|(Tx_0)-\lambda |x^*|(x_0)=0$. Hence $x_0$ being quasi-interior yields $T^*|x^*|=\lambda |x^*|$. It follows that $\lambda |x^*|=r(T) |x^*|=T^*|x^*|$. In particular, $\lambda =r(T)$.
\end{proof}

Note that for a positive operator $T$, both $[T\rangle$ and $\langle T]$ are multiplicative semigroups. We also need the following fact.
Suppose that $T$ and $S$ are two semi-commuting positive operators. If $TS\leq ST$ then $(TS)^n\leq S^nT^n$ for all $n$; if $ST\leq TS$ then $(ST)^n\leq T^nS^n$ for all $n$. In either case, we have $r(TS)=r(ST)\leq r(T)r(S)$.
\begin{lemma}\label{turo}Let $K>0$ be compact such that either $[K\rangle$ or $\langle K]$ is ideal irreducible. Then $r(K)>0$.
 \end{lemma}
\begin{proof}Suppose first that $[K\rangle$ is ideal irreducible and $r(K)=0$.
Let $\mathcal{K}$ be the ideal generated by $K$ in the semigroup $[K\rangle$. Then each member in $\mathcal{K}$ has the form $S_1KS_2$ for some $S_i\in [K\rangle$.
Clearly, $S_1KS_2\leq S_1S_2K$. Since $S_1S_2\in[K\rangle$, we have $r(S_1KS_2)\leq r(S_1S_2K)\leq  r(K)r(S_1S_2)=0$ by the remark preceding this lemma.
Therefore, $\mathcal{K}$ consists of quasi-nilpotent compact operators. Hence, $\mathcal{K}$ is ideal reducible by Drnov\v{s}ek's Theorem (see Corollary~10.47, \cite{abr}), 
and thus so is $[K\rangle$, by Proposition~2.1 of \cite{drn1}. This contradicts our assumption. Similar arguments work for the other case.
\end{proof}

\section{Main Results}\label{sec3}
For a compact operator $K>0$ such that either $[K\rangle$ or $\langle K]$ is ideal irreducible, we have $r(K)>0$ by Lemma~\ref{turo}. We will usually scale it so that $r(K)=1$.

\begin{theorem}\label{hyper}Suppose $K>0$ is compact, and either $[K\rangle$ or $\langle K]$ is ideal irreducible. Assume $r(K)=1$. Let $P$ be the spectral projection for $\sigma_{per}(K)$.
\begin{enumerate}
\item\label{hyperi1} There exist disjoint positive vectors $\{x_i\}_1^r$ and disjoint positive functionals $\{x_i^*\}_1^r$, where $r=\mathrm{rank}(P)$, such that
$$P=\sum_1^rx_i^*\otimes x_i;\quad x_i^*(x_j)=\delta_{ij},\;\forall\,i,j.$$
\item\label{hyperi2} $K|_{PX}$ is a
 permutation on $\{x_i\}_1^r$ and $K^*|_{P^*X^*}$ is a permutation on $\{x_i^*\}_1^r$. In particular, there exists $m\geq 1$ such that $P=\lim_n K^{nm}$.
\item\label{hyperi3} There exist a quasi-interior point $x_0>0$ and a strictly positive functional $x_0^*>0$ such that $$Kx_0= x_0\mbox{\;and\;}K^*x_0^*=x_0^*.$$
\item\label{hyperi4} For any $x>0$ and $x^*>0$,
$$0<\inf_n ||K^nx||\leq \sup_n||K^nx||<\infty,$$
$$0<\inf_n ||K^{n*}x^*||\leq \sup_n||K^{n*}x^*||<\infty.$$
In particular, $\lim_n||K^nx||^{\frac{1}{n}}=1$ for all $x>0$ and $\lim_n||K^{n*}x^*||^{\frac{1}{n}}=1$ for all $x^*>0$.
\item\label{hyperi5}Every operator semi-commuting with $K$ commutes with $K$. In particular,
$$[K\rangle=\langle K]=L_+(X)\cap \{K\}'.$$
\item\label{hyperi6}For any $0<S\leftrightarrow K$, there exist $\lambda_S\geq 0$, $x> 0$ and $x^*>0$ such that 
$$Sx=\lambda_S x,\quad S^*x^*=\lambda_S x^*.$$ 
    \end{enumerate}
\end{theorem}

\begin{proof}Suppose first that $[K\rangle$ is ideal irreducible.
We apply Lemma~\ref{pRiesz} to $K$. We claim that the nilpotent case is impossible. Indeed, otherwise, by Lemma~\ref{pRiesz}, $c_jK^{n_j}$ converges to a non-zero finite-rank nilpotent operator $N$ for some 
$n_j\uparrow\infty$ and positive reals $c_j\downarrow 0$. 
Clearly, $N$ is positive and compact. Thus Lemma~\ref{turo} implies that $[N\rangle$ is ideal reducible. It is easy to verify that $[K\rangle\subset [N\rangle$. Hence $[K\rangle$ is also ideal reducible, a contradiction.
Therefore, using Lemma~\ref{pRiesz} again, we have $P=\lim_jK^{n_j}$ for some $n_j\uparrow\infty$. In particular, $P>0$.
Therefore, the range $PX$ is a lattice subspace of $X$ with lattice operations $x\overset{*}\wedge y=P(x\wedge y)$ and $x\overset{*}\vee
y=P(x\vee y)$ for any $x,y\in PX$; see Proposition~11.5 on p.~214 of \cite{scha3}. Being a finite-dimensional Archimedean vector lattice, it is lattice
isomorphic to $\mathbb R^r$ with the standard order; see Corollary~1 on p.~70 of \cite{scha3}. Therefore, we can take positive *-disjoint vectors $x_i\in PX$
($i=1,\dots,r$) and positive $y_i^*\in(PX)^*$ such that $y_i^*(x_j)=\delta_{ij}$. Put $x_i^*=y_i^*\circ P$. Then $x_i^*\in X^*_+$, $P=\sum_1^rx_i^*\otimes x_i$ and $x_i^*(x_j)=\delta_{ij}$.\par

Being a spectral subspace, $PX$ is invariant under $K$. Note that $K|_{PX}$ is positive on the lattice subspace $PX$. It follows from $P=\lim_jK^{n_j}>0$ that $I|_{PX}=\lim_j(K|_{PX})^{n_j}$. Thus $(K|_{PX})^{-1}=\lim_j (K|_{PX})^{n_j-1}$ is also positive on $PX$. It is well known that a
positive operator on $\mathbb{R}^r$ has a positive inverse if and only if it is a weighted permutation on the standard basis with positive weights, if and only if it is a direct sum of weighted cyclic permutations
with positive weights. 
Since $\sigma(K|_{PX})=\sigma_{per}(K)\subset \{z\in\mathbb{C}:|z|=1\}$, it is easily seen that after appropriately scaling the basis vectors $x_i$'s, $K|_{PX}$ is a permutation on $x_i$'s. We accordingly scale $x_i^*$'s so that we still have $x_i^*(x_j)=\delta_{ij}$ and $P=\sum_1^rx_i^*\otimes x_j$.
Note that $Kx_j=KPx_j=PKx_j=\sum_{i=1}^kx_i^*(Kx_j)x_i$ and that $K^*x_j^*=K^*P^*x_j^*=P^*K^*x_j^*=\sum_{i=1}^kx_j^*(Kx_i)x_i^*$. Hence the matrix of $K^*|_{P^*X^*}$ relative to $\{x_i^*\}$ 
is the transpose of the matrix of $K|_{PX}$ relative to $\{x_i\}$, and thus $K^*|_{P^*X^*}$ is a permutation on $x_i^*$'s. Put $x_0=\sum_1^rx_i$ and $x_0^*=\sum_1^rx_i^*$. Then it is clear that $Kx_0=x_0$ and $K^*x_0^*=x_0^*$.\par

Since $K|_{PX}$ is a permutation on $x_i$'s, we can take $m\geq 1$ such that $(K|_{PX})^{m}=I|_{PX}$. Denote by $Q$ the spectral projection of $K$ for 
$\sigma(K)\backslash\sigma_{per}(K)$. Then $r(K|_{QX})<1$. Thus $(K|_{QX})^{n}\rightarrow 0$ as $n\rightarrow\infty$. It follows that $K^{mn}=(K|_{PX})^{mn}\oplus (K|_{QX})^{mn}\rightarrow I|_{PX}\oplus 0=P$ as $n\rightarrow\infty$.

Since $P=\sum_1^rx_i^*\otimes x_i$, it is easy to see that $P$ is strictly positive if and only if $x^*_0$ is strictly positive, and that $P^*$ is strictly positive if and only if $x_0$ is quasi-interior.
We now prove that both $P$ and $P^*$ are strictly positive. It is easy to verify that $N_P:=\{x\in X:P|x|=0\} $ is a closed ideal invariant under $[P\rangle$. Since $[K\rangle \subset [P\rangle$, we know $N_P$ is also invariant under $[K\rangle$. From this it follows easily that $N_P=\{0\} $. Thus $P$ is strictly positive, and so is $x_0^*$. 
Now for any $T\in[K\rangle$, we have $x_0^*((TK-KT)x_0)=x_0^*(TKx_0)-(K^*x_0^*)(Tx_0)=0$. By strict positivity of $x_0^*$, we have $KTx_0=TKx_0=Tx_0$. 
Thus $Tx_0\in\ker(1-K)\subset PX\subset \overline{I_{x_0}}$. This implies that $\overline{I_{x_0}}$ is invariant under $[K\rangle$, hence $\overline{I_{x_0}}=X$. It follows that $x_0$ is quasi-interior and thus $P^*$ is strictly positive.\par

Since $P$ is strictly positive, it follows from $0=x_i\overset{*}\wedge x_j=P(x_i\wedge x_j)$ that $x_i\perp x_j$ whenever $i\neq j$.
We now prove that $x_i^*$'s are disjoint.
Indeed, by Riesz-Kantorovich formulas, for $i\neq j$,
$$0\leq (x_i^*\wedge x_j^*)(x_0)=\inf_{0\leq u\leq x_0}\{x_i^*(u)+x_j^*(x_0-u)\}\leq x_i^*(x_j)+x_j^*(x_0-x_j)=0.$$
Thus $(x_i^*\wedge x_j^*)(x_0)=0$, yielding that $x_i^*\wedge x_j^*=0$, since $x_0$ is quasi-interior. This proves \eqref{hyperi1}, \eqref{hyperi2} and \eqref{hyperi3}. \eqref{hyperi5} follows from \eqref{hyperi3} and Lemma~\ref{scc}.\par

For \eqref{hyperi4}, fix any $x>0$. Since $P$ is strictly positive, we have $Px>0$. Now $K|_{PX}$ being a permutation implies that
$$0<\liminf_{n}||(K|_{PX})^nPx||\leq \limsup_n||(K|_{PX})^nPx||<\infty.$$
Recall that $Q$ is the spectral projection of $K$ for 
$\sigma(K)\backslash\sigma_{per}(K)$ and $(K|_{QX})^{n}\rightarrow 0$. It follows from $K^nx=(K|_{PX})^nPx+(K|_{QX})^nQx$ that
$$0<\liminf_{n}||K^nx||\leq \limsup_n||K^nx||<\infty.$$
Hence, $K$ is strictly positive. This in turn implies 
$$0<\inf_{n}||K^nx||\leq \sup_n||K^nx||<\infty.$$
Taking the $n$-th root, we have $\lim_n||K^nx||^{\frac{1}{n}}=1$ for all $x>0$.
For the dual case, a similar argument works.\par

For \eqref{hyperi6}, pick any $0<S\in\{K\}'$. Then $SP=PS$. Note that the matrix of $S|_{PX}$ relative to $\{x_i\}$ is $(x_i^*(Sx_j))_{i,j}$, 
and that the matrix of $S^*|_{P^*X^*}$ relative to $\{x_i^*\}$ is $(x_i^*(Sx_j))_{j,i}$, which is the transpose of $(x_i^*(Sx_j))_{i,j}$. Since both matrices are positive, 
they have positive eigenvectors for the spectral radius. It follows that there exist $0< x\in PX$ and $0<x^*\in P^*X^*$ such that $Sx=r(S|_{PX})x$ and $S^*x^*=r(S|_{PX})x^*$.\par
\smallskip
Now assume that $\langle K]$ is ideal irreducible. We shall apply similar arguments. In fact, we only need to modify the proof of strict positivity of $P$ and $P^*$.
It is easy to verify that the ideal ${I_{PX}}\neq \{0\}$ is invariant under $\langle P]$ and thus is invariant under $\langle K]$. Therefore, $\overline{I_{PX}}=X$.
On the other hand, we clearly have ${I_{PX}}={I_{x_0}}$. Hence $\overline{I_{x_0}}=X$. It follows that $x_0$ is quasi-interior and thus $P^*$ is stricly positive.
Now for any $T\in\langle K]$, $\langle(K^*T^*-T^*K^*)x_0^*,x_0\rangle=x_0^*(TKx_0)-x_0^*(KTx_0)=0$. Hence $K^*T^*x_0^*=T^*K^*x_0^*=T^*x_0^*$. This implies $T^*x_0^*\in \ker(1-K^*)\subset P^*X^*=\mathrm{Span}\{x_i^*\}_1^r$.
We claim that $x_0^*$ is strictly positive. Suppose, otherwise, $x_0^*(x)=0$ for some $x>0$. Then $x_i^*(x)=0$ for $1\leq i\leq r$. It follows that $x_0^*(Tx)=T^*x_0^*(x)=0$ for any $T\in \langle K]$.
By Proposition~2.1 of \cite{drn1}, $\langle K]$ is ideal reducible, a contradiction.
It follows that $x^*_0$ and $P$ are both strictly positive. 
\end{proof}

\begin{remark}\label{rem}We apply Theorem~\ref{hyper} to the operator $K$ in Theorem~\ref{dsns}. By Lemma~\ref{turo}, $r(K)>0$. We scale $K$ so that $r(K)=1$. Then $K|_{PX}$ is a permutation on $x_i$'s. We claim that it is a cyclic permutation.  Suppose, otherwise, $K|_{PX}$ has a cycle of length $m<r$. Without loss of generality, assume that $K|_{PX}$ has a cyclic permutation on $x_1,\dots,x_m$. Then the closed ideal $\overline{I_{\{x_1,\dots,x_m\}}}$ is non-zero and invariant under $K$. Since $\overline{I_{\{x_1,\dots,x_m}\}}$ is disjoint from $x_r$, it is proper, contradicting ideal irreducibility of $K$. This proves the claim. Theorem~\ref{dsns} now follows immediately.
\end{remark}

\begin{example}\label{remi2}
Consider $T=\left[
    \begin{smallmatrix}
      1 & 1 & 1 & 1 \\
       1 &1  & 1 &1  \\
      1 & 1 & 1 &1  \\
     1 & 1 & 1 & 1 \\
    \end{smallmatrix}
  \right]$ and $ K=\left[
    \begin{smallmatrix}
      0 & 1 &  &  \\
     1 & 0 &  &  \\
     &  & 1 & 0 \\
    &  & 0 & 1 \\
    \end{smallmatrix}
  \right]$.
Then the peripheral spectral projection of $K$ is the identity,  $T\leftrightarrow K$, and $T$ is irreducible (in particular, $[K\rangle$ and $\langle K]$ are both irreducible). Modifying this example, we can easily see that, for the operator $K$ in Theorem~\ref{hyper}, $K|_{PX}$ could be an arbitrary permutation.
\end{example}
\begin{example}\label{remi3} We can not expect $\lambda>0$ in Theorem~\ref{hyper}\eqref{hyperi6}. Let $K$ be as in Example~\ref{remi2}. Put
$S=\left[
    \begin{smallmatrix}    
      0& 0 &  &  \\
      0 &0  &  &  \\
       &  & 0 & 1 \\
      &  &  0& 0 \\
    \end{smallmatrix}
  \right]
$.
Then $S\leftrightarrow K$, and $S$ is nilpotent.
\end{example}

Theorem~\ref{hyper} immediately yields the following corollary.

\begin{corollary}\label{pcu1}Suppose $T$ and $K$ are two non-zero positive semi-commuting operators such that $T$ is ideal irreducible and $K$ is compact. Then $TK=KT$. 
\end{corollary}

We will provide an alternative proof of this corollary which is of independent interest.
Recall that a positive operator $T $ is 
called \textbf{strongly expanding} if it sends non-zero positive vectors to quasi-interior points. It is known that if $T$ is ideal irreducible 
then $\sum_1^\infty\frac{T^n}{\lambda^n}$ is strongly expanding for all $\lambda>r(T)$ (see Corollary~9.14, \cite{abr}). We also need the following \emph{comparison theorem}.

\begin{lemma}[\cite{gao}]\label{compa}
Suppose $A\geq B\geq 0$, $r(A)=r(B)$ and $A$ is ideal irreducible. If $A$ and $B$ are both compact then $A=B$.
\end{lemma}

\begin{proof}[A second proof of Corollary~\ref{pcu1}] Without loss of generality, assume $||T||<1$. Then $\widetilde{T}=\sum_1^\infty T^n$ is strongly expanding. We claim that $K\widetilde{T}$ and $\widetilde{T}K$ do not have a 
common non-trivial invariant closed ideal. Otherwise, let $J$ be such an ideal. If there exists $0<x\in J$ such that $Kx>0$, then $\widetilde{T}(Kx)$ is a quasi-interior point. But 
$\widetilde{T}Kx\in \widetilde{T}K(J)\subset J$ implies $J=X$, a contradiction. Hence, $J\subset\ker K$. Now pick any $0<x\in J$. Then $\widetilde{T}x>0$ is a quasi-interior point. 
From $K(K\widetilde{T}(x))\in K(K\widetilde{T}(J))\subset K(J)=\{0\}$, it follows that $K^2=0$. In particular, $r(K)=0$, contradicting ideal irreducibility of $T$ by Lemma~\ref{turo}. This proves the claim.\par
Now assume $TK\geq KT$. Then $T^nK\geq KT^n$ for all $n\geq 1$. Thus $\widetilde{T}K\geq K\widetilde{T}\geq0$. If $\widetilde{T}K$ has a non-trivial invariant closed ideal then it is also invatiant under $K\widetilde{T}$, contradicting the preceding claim. Thus $\widetilde{T}K$ is ideal irreducible.
Note now that $\widetilde{T}K$ and $K\widetilde{T}$ are both compact, and that $r(\widetilde{T}K)=r(K\widetilde{T})$. Thus $\widetilde{T}K=K\widetilde{T}$ by Lemma~\ref{compa}. 
This immediately implies that $TK=KT$. For $TK\leq KT$, we have $K\widetilde{T}\geq \widetilde{T}K$. The same argument yields $K\widetilde{T}=\widetilde{T}K$ and $TK=KT$.\end{proof}

We refer to \cite{gao} for more comparison theorems. We look at the operator $T$ in Corollary~\ref{pcu1} more closely.

\begin{proposition}\label{peris}Let $T>0$ be ideal irreducible and $K>0$ be compact. Suppose that $T\leftrightarrow K$.
\begin{enumerate}
 \item\label{perisi1}There exist $0<\lambda\leq r(T)$, a quasi-interior point $x_0>0$ and a strictly positive functional $x_0^*>0$ such that
$$Tx_0=\lambda x_0,\;T^*x_0^*=\lambda x_0^*;\; \;Kx_0=r(K) x_0, \;K^*x_0^*=r(K)x_0^*.$$
Moreover, $\lambda=r(T)$ if $r(T) $ is an eigenvalue of either $T$ or $T^*$.
\item\label{perisi2} $\liminf_n||T^nx||^{\frac{1}{n}}\geq \lambda$ and $\liminf_n||T^{n*}x^*||^{\frac{1}{n}}\geq \lambda$ for any $x>0$ and $x^*>0$.
\item\label{perisi3}Every operator semi-commuting with $T$ commutes with $T$. In particular,
$$[T\rangle=\langle T]=L_+(X)\cap \{T\}'.$$
\item\label{perisi4} For any $S\in \{T\}'$, $Sx_0=\lambda_Sx_0$ for some $\lambda_S\in\mathbb{R}$. If, in addition, $S>0$, then $r(S)>0$.
\end{enumerate}
\end{proposition}
\eqref{perisi1} and \eqref{perisi2} have been proved in \cite{gt2} via semigroup techniques. Here we provide a direct elementary proof.
\begin{proof}
Without loss of generality, assume $||T||<1$. Recall that since $T$ is ideal irreducible, $\sum_1^\infty T^n$ is strongly expanding. Hence so is $\widetilde{K}:=(\sum_1^\infty T^n)K(\sum_1^\infty T^n)$. In particular, 
$\widetilde{K}$ is an ideal irreducible compact operator.
Applying Theorem~\ref{dsns}\eqref{dsnsi1} to $\widetilde{K}$, we obtain a quasi-interior point $x_0>0$ and a strictly positive functional $x_0^*>0$ such that
$$\ker(r(\widetilde{K})-\widetilde{K})=\mathrm{Span}\{x_0\}\mbox{\;and\;}
\ker(r(\widetilde{K})-\widetilde{K}^*)=\mathrm{Span}\{x_0^*\}.$$
\indent Since $T\leftrightarrow \widetilde{K}$, the one-dimensional spaces $\ker(r(\widetilde{K})-\widetilde{K})$ and $\ker(r(\widetilde{K})-\widetilde{K}^*)$ are invariant under $T$ and $T^*$, respectively. 
Thus there exist $\lambda,\delta\in \mathbb{R}$ such that $$Tx_0=\lambda x_0\mbox{\;and\;}T^*x_0^*=\delta x_0^*.$$
Note that $\delta x_0^*(x_0)=T^*x_0^*(x_0)=x_0^*(Tx_0)=\lambda x_0^*(x_0)$. Hence $\delta=\lambda$. Since $T>0$ can not vanish on quasi-interior points, we have $\lambda>0$. The ``moreover'' part in \eqref{perisi1} as well as \eqref{perisi2} follows from Lemma~\ref{locaq}.\par
Since $K\leftrightarrow \widetilde{K}$, a similar argument yields $0<\mu\leq r(K)$ such that $$Kx_0=\mu x_0\mbox{\;and\;}K^*x_0^*=\mu x_0^*.$$
By Krein-Rutman Theorem, $r(K)$ is an eigenvalue of $K$. Hence Lemma~\ref{locaq} implies $\mu=r(K)$.\par

\eqref{perisi3} follows from Lemma~\ref{scc}.\par

\eqref{perisi4} Using the same argument as in Theorem 5.2, p.~329, \cite{scha3}, one can easily see that $\ker(\lambda-T)=\mathrm{Span}\{x_0\}$. Since $S\leftrightarrow T$, we know $Sx_0\in\ker(\lambda-T)$, and thus $Sx_0=\lambda_S x_0$ for some $\lambda_S\in\mathbb{R}$. If $S>0$, 
then $Sx_0>0$, thus $r(S)\geq \lambda_S>0$.
\end{proof}

\begin{example}In Proposition~\ref{peris}\eqref{perisc3i1}, $x_0$ and $x_0^*$ depend on $T$. Let $T$ and $K$ be as in Example~\ref{remi2}. It is easy to see that the positive eigenvectors of $T$ are exactly all the positive scalar multiples of $(1,1,\cdots,1)^t$.
Let $R=\left[
    \begin{smallmatrix}
      1 & 1 & 1 & 1 \\
       1 &1  & 1 &1  \\
      1 & 1 & \frac{5}{2} &\frac{5}{2}  \\
     1 & 1 & \frac{5}{2}& \frac{5}{2} \\
    \end{smallmatrix}
  \right]$. Then $R$ is also irreducible and commutes with $K$. But the positive eigenvectors of $R$ are positive scalar multiples of $(1,1,2,2)^t$.
\end{example}

\begin{remark}Sirotkin~\cite{relom} proved a Lomonosov-type theorem for positive operators on real Banach lattices, which implies that if $T>0$ is non-scalar and $K>0$ is compact and $S\leftrightarrow T\leftrightarrow K$, 
then $S$ has a non-trivial invariant closed subspace; cf.~Corollary~2.4, \cite{relom}. Proposition~\ref{peris}\eqref{perisi4} implies that, in such a chain, if $S$ is also non-scalar, then
either $T$ has a non-zero proper invariant closed ideal, or $S$ has a non-trivial hyperinvariant closed subspace (namely, the eigenspace of $S$ for $\lambda_S$).
\end{remark}

We end this section with the following proposition.

\begin{proposition}\label{perisc3}Suppose $T>0$ is ideal irreducible and $K>0$ is compact.
\begin{enumerate}
\item\label{perisc3i1} There exist a quasi-interior point $x_0>0$ and a strictly positive functional $x_0^*$ such that for any $S\in L(X)$ with $T\leftrightarrow S\leftrightarrow K$, $$Sx_0=\lambda_S x_0,\;\;S^*x_0^*=\lambda_S x_0^*,$$
for some $\lambda_S\in\mathbb{R}$.
\item\label{perisc3i2} If, in addition, $S>0$, then $\liminf_n||S^nx||^{\frac{1}{n}}\geq \lambda_S$ and $\liminf_n||S^{n*}x^*||^{\frac{1}{n}}\geq \lambda_S$ for any $x>0$ and $x^*>0$.
In particular, $r(S)>0$. Moreover, $\lambda_S=r(S)$ if $r(S)$ is an eigenvalue of either $S$ or $S^*$.
\item\label{perisc3i3} Every operator semi-commuting with $S$ commutes with $S$.
\end{enumerate}
\end{proposition}

\begin{proof}Without loss of generality, assume $||T||<1$. As before, it is easily seen that $\widetilde{K}:=\left(\sum_1^\infty T^n\right)K\left(\sum_1^\infty T^n\right)$ is a compact ideal irreducible operator. 
Thus by Theorem~\ref{dsns}\,\eqref{dsnsi1}, there exist a quasi-interior point $x_0>0$ and a strictly positive functional $x_0^*>0$ such that
$$\ker(r(\widetilde{K})-\widetilde{K})=\mathrm{Span}\{x_0\}\mbox{\;and\;}\ker(r(\widetilde{K})-\widetilde{K}^*)=\mathrm{Span}\{x_0^*\}.$$
Since $S\leftrightarrow \widetilde{K}$, these one-dimensional spaces are invariant under $S$ and $S^*$, respectively. From this (\ref{perisc3i1}) follows easily.\par
\eqref{perisc3i2} follows from Lemma~\ref{locaq}. 
\eqref{perisc3i3} follows from Lemma~\ref{scc}.
\end{proof}

Clearly, \eqref{perisc3i2} improves Theorem~\ref{aab1}.

\begin{remark}\label{remark11}Note that the operator $K$ in Theorem~\ref{hyper} can be replaced with a peripherally Riesz operator $R>0$. The same proof goes along.
In Propositions~\ref{peris} and \ref{perisc3}, the operator $K$ can also be replaced with a peripherally Riesz operator $R>0$. Simply note that Lemma~\ref{pRiesz} yields a compact operator to take the position of $R$.
\end{remark}

\section{Applications}\label{sec4}

Using the results in the previous section, we can establish a few interesting criteriors on ideal reducibility. The following is immediate by Lemma~\ref{turo} and Theorem~\ref{hyper}.

\begin{proposition}\label{app1}Let $K>0$ be compact. Then $[K\rangle$ and $\langle K]$ both have non-trivial invariant closed ideals if any of the following is satisfied:
\begin{enumerate}
 \item\label{app1i1} $r(K)=0$, or $\liminf_n||K^nx||^{\frac{1}{n}}<r(K)$ for some $x>0$, or $\liminf_n||K^{n*}x^*||^{\frac{1}{n}}<r(K)$ for some $x^*>0$;
\item there exists $S\in L(X)$ such that either $SK>KS$ or $SK<KS$.
\end{enumerate}
\end{proposition}

Proposition~\ref{app1}\eqref{app1i1} clearly improves Theorem~\ref{drnl}.\par
Recall that if $T>0$ and $S>0$ semi-commute then $r(TS)\leq r(T)r(S)$.

\begin{proposition}\label{sver}Suppose $T$ and $K$ are two non-zero semi-commuting positive operators such that $K$ is compact. Then $T$ has non-trivial invariant closed ideals if any of the following is satisfied:
\begin{enumerate}
\item\label{sveri1} $r(K)=0$, or $\liminf_n||K^nx||^{\frac{1}{n}}<r(K)$ for some $x>0$, or $\liminf_n||K^{n*}x^*||^{\frac{1}{n}}<r(K)$ for some $x^*>0$;
\item\label{sveri2} $r(TK)=0$, or $\liminf_n||T^nx||^{\frac{1}{n}}<\frac{r(TK)}{r(K)}$ for some $x>0$, or $\liminf_n||T^{n*}x^*||^{\frac{1}{n}}<\frac{r(TK)}{r(K)}$ for some $x^*>0$;
\item\label{sveri3} there exists a quasi-nilpotent operator $S>0$ semi-commuting with $T$;
\item\label{sveri4} there exists $ S\in L(X)$ such that $ST<TS$ or $ST>TS$.
    \end{enumerate}
\end{proposition}

\begin{proof}We prove by way of contradiction. Assume that $T$ is ideal irreducible. Then $TK=KT$ by Corollary~\ref{pcu1}. \eqref{sveri1} is clear by Proposition~\ref{app1}.
It is easy to verify that $TK>0$. Since $TK\leftrightarrow T$, we have $r(TK)>0$ by Proposition~\ref{peris}\eqref{perisi4}. Replacing $K$ with $TK$ in Proposition~\ref{peris}, we have
$$Tx_0=\lambda x_0,\;T^*x_0^*=\lambda x_0^*;\; TKx_0= r(TK)x_0, \;(TK)^*x_0^*=r(TK)x_0^*.$$
Applying Lemma~\ref{locaq} to $TK$, we have $$r(TK)=\lim_n||(TK)^nx||^{\frac{1}{n}}\leq r(K)\liminf_n||T^nx||^{\frac{1}{n}}$$ for any $x>0$. The dual case can be proved similarly. This proves \eqref{sveri2}.
\eqref{sveri3} follows from Lemma~\ref{scc} and Proposition~\ref{peris}\eqref{perisi4}.
\eqref{sveri4} also follows from Lemma~\ref{scc}.
\end{proof}
Proposition~\ref{sver}\eqref{sveri1} clearly improves Theorem~\ref{aab2}.\par
\medskip
In \cite{abr}, it is proved that Theorem~\ref{aab2} remains true if $K$ is merely assumed to dominate a non-zero AM-compact operator; see Theorems~10.25 and 10.26 of \cite{abr}. The authors also asked the following natural question. If $K$ dominates a non-zero AM-compact operator, then in case \eqref{aab2i2} of Theorem~\ref{aab2}, can we replace the quasinilpotency of $K$ by the local quasinilpotency at a positive vector? The following example shows that it generally fails even when $K$ is a non-zero positive AM-compact operator. Let $X=\ell_2$. Let $L$ and $R$ be the left and right shifts on $X$, respectively. Let $T=L+R$. Since the order intervals in $\ell_2$ are compact, both $T $ and $L$ are AM-compact. It is straightforward verifications that $TL<LT$, $Le_1=0$ where $e_1=(1,0,0,\cdots)$, but $T$ is ideal irreducible. Surprisingly, the question has an affirmative answer when $K$ is a compact operator; this follows from our Proposition~\ref{sver}\eqref{sveri1}.
\medskip

We now turn to another application of Corollary~\ref{pcu1}.
In \cite{jbr}, it is proved that if two operators $T>0$ and $S>0$ semi-commute and are both compact, then their commutator $TS-ST$ is quasi-nilpotent. It is also shown there that there exist two semi-commuting 
operators $T>0$ and $S>0$, neither of which is compact, such that $TS-ST$ is not quasi-nilpotent. So the authors asked the following question.

\begin{theorem2}
Suppose $T$ and $K$ are two semi-commuting positive operators such that $K$ is compact. Is $TK-KT$ necessarily quasi-nilpotent?
\end{theorem2}

Theorem\;3.6 in \cite{drn2} gave a partial solution of this question by proving that the commutator is indeed quasi-nilpotent provided that, in addition, it semi-commutes with $K$.\par

We now answer this question and prove that it is generally true. To this end, we need to recall some necessary notions. A collection $\mathcal{C}$ of closed subspaces of $X$ is called a chain if it is totally ordered under inclusion. 
For any $M\in\mathcal{C}$, the predecessor $M_-$ of $M$ in $\mathcal{C}$ is defined to be the closed linear span of all proper closed subspaces of $M$ that belong to $\mathcal{C}$.
The following lemma is straightforward to verify.

\begin{lemma}\label{mcha}Let $\mathcal{C}$ be a chain of closed ideals of $X$, $M\in\mathcal{C}$.
Then $M_-$ is a closed ideal of $X$, $M_-\subset M$ and $\mathcal{C}\cup\{M_-\}$ is chain.
\end{lemma}

\begin{lemma}[\cite{drn2000}]\label{mchb}Let $\mathcal{C}$ be a chain of closed ideals of $X$.
Then it is maximal as a chain of closed subspaces of $X$ if and only if it is maximal as a
chain of closed ideals of $X$.\end{lemma}

Recall that a collection $\mathcal{S}$ of positive operators is called \textbf{ideal triangularizable} if there exists a chain of closed ideals of $X$ such that each member in the chain is invariant under $\mathcal{S}$ and 
the chain itself is maximal as a chain of closed subspaces of $X$ (cf.~Lemma~\ref{mchb}). Such a chain is called an ideal triangularizing chain for $\mathcal{S}$.

\begin{theorem}\label{quasi}
Suppose $T$ and $K$ are two non-zero positive semi-commuting operators such that $K$ is compact. Then $S:=TK-KT$ is quasi-nilpotent.
\end{theorem}

\begin{proof}Since replacing $T$ with $T+K$ does not change the commutator, we can assume $T\geq K>0$. Let $\mathcal{C}$ be a maximal chain of invariant closed ideals of $T$ (existence of such a chain follows from Zorn's lemma). 
Take any $M\in \mathcal{C}$. It is easily seen that $M_-$ is invariant under $T$. Hence, by Lemma~\ref{mcha},  $M_-\in\mathcal{C}$. 
We claim that the induced quotient operator $\widetilde{T}$ on $M/M_-$ is ideal irreducible. 
Suppose that, otherwise, $J$ is a non-trivial closed ideal of $M/M_-$ invariant under $\widetilde{T}$. We consider $\pi^{-1}(J)=\{x\in M:[x]\in J\}$, where $\pi$ is the canonical quotient mapping from $M$ onto $M_-$. 
By Proposition~1.3, p.~156, \cite{scha3}, 
$\pi^{-1}(J)$ is a closed ideal of $M$, hence is a closed ideal of $X$. It is clearly invariant under $T$, properly contains $M_-$ and is properly contained in $M$. Thus it is easily seen that $\pi^{-1}(J)$ is comparable 
with members of $\mathcal{C}$. But $\pi^{-1}(J)\not\in\mathcal{C}$, contradicting maximality of $\mathcal{C}$.\par

It follows that $\widetilde{T}$ is ideal irreducible on $M/M_-$. Since $T\geq K\geq 0$, both $M$ and $M_-$ are invariant under $K$; hence the quotient operator $\widetilde{K}$ is well defined on $M/M_-$. 
Corollary~\ref{pcu1} implies $\widetilde{S}=\widetilde{T}\widetilde{K}-\widetilde{K}\widetilde{T}=0$.\par

For each $M\in \mathcal{C}$, let $\widetilde{\mathcal{C}_M}$ be a maximal chain of closed ideals of $M/M_-$ (existence of such a chain follows from Zorn's lemma). Put $\mathcal{C}_M=\{\pi^{-1}(J):J\in\widetilde{\mathcal{C}_M}\}$. 
Then $\mathcal{C}_M$ consists of closed ideals of $X$ each of which contains $M_-$ and is contained in $M$. Since $\widetilde{S}=0$ on $M/M_-$, each member of $\mathcal{C}_M$ is invariant under $S$. Thus it is easily seen that 
$\mathcal{C}_1=\mathcal{C}\cup_{M\in\mathcal{C}}\mathcal{C}_M$ is a \textbf{chain} of closed ideals of $X$ each of which is invariant under $S$.\par

We claim that $\mathcal{C}_1$ is an ideal triangularizing chain for $S$. It remains to prove $\mathcal{C}_1$ is maximal as a chain of closed subspaces of $X$.
Suppose, otherwise, there exists a closed subspace $Y\notin\mathcal{C}_1$ such that $\mathcal{C}_1\cup\{Y\}$ is still a chain. Consider $M:=\cap_{J\in \mathcal{C},J\supset Y}J$ and $N:=\overline{\cup_{J\in \mathcal{C},J\subset Y}J}$. 
They are well defined since $\{0\},X\in\mathcal{C}$. It is easily seen that they are closed ideals of $X$ invariant under $T$. Since each member of $\mathcal{C}$ is comparable with $Y$, 
it is easy to see that each $J\in\mathcal{C}$ either is contained in $N$ or contains $M$.  Note also that $N\subset Y\subset M$. Hence, by maximality of $\mathcal{C}$, we have $M,N\in\mathcal{C}$. 
It also follows that $N=M$ or $M_-$, the predecessor of $M$ in $\mathcal{C}$. 
The first case is impossible, since it forces $Y=N=M\in\mathcal{C}\subset \mathcal{C}_1$. Hence, $M_-=N\subsetneqq Y\subsetneqq M$. Note that $Y/M_-$ is a closed subspace of $M/M_-$. 
Clearly, $Y/M_-\notin\widetilde{\mathcal{C}_M} $. Since $Y$ is comparable with each member of $\mathcal{C}_1$, $Y/M_-$ is comparable with each member of $\widetilde{\mathcal{C}_M}$, contradicting maximality of 
$\widetilde{\mathcal{C}_M}$, by Lemma~\ref{mchb}. This proves the claim.\par

Now note that for any $N\in \mathcal{C}_1$, we can find $M\in \mathcal{C}$ such that $M_-\subset N_-\subset N\subset M$, where $M_-$ is the predecessor of $M$ in $\mathcal{C}$ and $N_-$ is the predecessor 
of $N$ in $\mathcal{C}_1$. Since $\widetilde{S}=0$ on $M/M_-$, $\widetilde{S}=0$ on $N/N_-$. Hence, by Ringrose's theorem (Theorem~7.2.3, p.~156, \cite{rad}), $\sigma(S)=0$, i.e.~$S$ is quasi-nilpotent.
\end{proof}
\medskip
A recent preprint~\cite{drn2012} provides an independent proof of Theorem~\ref{quasi}.
\\
\medskip

\noindent\textbf{Acknowledgement.}
The author would like to thank the organizers of the conference ``Ordered Spaces and Applications'' held in Novermber 2011, National Technical University of Athens for the opportunity 
to present there the results of this paper.
The author is also grateful to his advisor Dr.~Vladimir G.~Troitsky for suggesting these problems and reading the manuscript.

\end{document}